\documentclass[11pt]{amsart}

\usepackage{GJNstyle}
\usepackage{fullpage}
\usepackage[all, cmtip]{xy}
\usepackage[foot]{amsaddr}

\title{Groups acting on CAT(0) cube complexes with uniform exponential growth}
\author{Radhika Gupta$^1$}
\address{$^1$ Temple University,
Department of Mathematics,
1805 N. Broad Street
Philadelphia, PA. 19122
}
\email{radhikagupta.maths@gmail.com}
\author{Kasia Jankiewicz$^2$}
\address{$^2$ University of California Santa Cruz,
Department of Mathematics,
1156 High Street,
Santa Cruz, CA 95064}
\email{kasia@ucsc.edu}
\author{Thomas Ng$^3$}
\address{$^3$ Technion -- Israel Institute of Technology,
Mathematics Department,
Haifa, Israel. 32000}
\email{thomas.ng.math@gmail.com}

\date{\today}

\begin{document}
\begin{abstract}

We study uniform exponential growth of groups acting on CAT(0) cube complexes.  We show that groups acting without global fixed points on CAT(0) square complexes either have uniform exponential growth or stabilize a Euclidean subcomplex.  This generalizes the work of Kar and Sageev that considers free actions.  Our result lets us show uniform exponential growth for certain groups that act improperly on CAT(0) square complexes, namely, finitely generated subgroups of the Higman group and triangle-free Artin groups.  We also obtain that non-virtually abelian groups acting freely on CAT(0) cube complexes of any dimension with isolated flats that admit a geometric group action have uniform exponential growth.   

\end{abstract}

\maketitle
\section{Introduction}

%
%

In this article, we continue the inquiry to determine which groups that act on CAT(0) cube complexes have uniform exponential growth. Let $G$ be a group with finite generating set $S$ and corresponding Cayley graph $\text{Cay}(G,S)$ equipped with the word metric. Let $B(n,S)$ be the ball of radius $n$ in $\text{Cay}(G,S)$. The \emph{exponential growth rate of $G$ with respect to $S$} is defined as $$ w(G,S) := \lim_{n \to \infty} |B(n,S)|^{1/n}.$$
The \emph{exponential growth rate of $G$} is defined as $$w(G):= \inf\braces{w(G,S) \;|\; S \text{ finite generating set}}.$$ 
We say a group $G$ has \emph{exponential growth} if $w(G,S) >1$ for some (hence every) finite generating set $S$. A group $G$ is said to have \emph{uniform exponential growth} if $w(G)>1$. The reader is referred to de la Harpe's book \cite{delaharpe} for more details on growth of groups. 

Many groups of exponential growth are known to have uniform exponential growth, for instance, non-elementary hyperbolic groups \cite{Koubi}, relatively hyperbolic groups \cite{Xie},  solvable groups \cite{Alperin,Osin:solvable},
nontrivial amalgamated free products and HNN-extensions \cite{BucherHarpe}, one-relator groups \cite{GrigorchukHarpe}, linear groups over a field of characteristic zero \cite{EMO}, and many hierarchically hyperbolic groups \cite{ANS}. The first examples of groups with exponential growth that do not have uniform exponential growth were introduced by Wilson \cite{Wilson}.  

Kar and Sageev showed that if a group acts freely on a CAT(0) \emph{square} complex, then either it has uniform exponential growth or it is virtually abelian \cite{KarSageev} . 
In this article, we generalize the result of Kar and Sageev by removing the assumption of a free action. 
By semisimplicity of cubical isometries \cite[Theorem~1.4]{Haglund:semisimplicity}, freeness of the action implies that the group is torsion-free.
Wise shows, however, that groups actings geometrically on CAT(0) cube complexes need not even be virtually torsion-free \cite[Section~9]{Wise:antitorus}. 
We show that even if the group contains elliptic elements, we can still get uniform exponential growth.

\begin{restatable}{thmAlph}{ThmA}\label{thm:2Dtorsion}
	Let $G$ be a finitely generated group acting \emph{without global fixed point} on a CAT(0) square complex $X$. Then either $G$ has uniform exponential growth with $w(G) \geq \sqrt[600]{2}$ or $G$ stabilizes a flat or line in $X$. 
\end{restatable}

See \Cref{sec:background} for the definition of a flat.  
In the setting of groups acting \emph{properly} on CAT(0) square complexes, stabilizing a flat can be upgraded to virtually abelian.

\begin{restatable*}{cor}{CorProper}\label{cor:Proper}
Let $G$ be a finitely generated group that acts \emph{properly} on a CAT(0) square complex. Then either $G$ has uniform exponential growth with $w(G) \geq \sqrt[600]{2}$, or $G$ is virtually abelian.
\end{restatable*}

Most known results on uniform exponential growth can be shown by producing a constant $M > 0$ such that in any generating set there exists a pair of elements with word length at most $M$ that generate a free semigroup or subgroup.

\begin{defn}[$N$-short subgroup]
	Let $G$ be a finitely generated group, with a finite generating set $S$. We say that a sub(semi)group $H$ in $G$ is \emph{$N$-short with respect to $S$}, if there exists a finite collection of words with $S$-length at most $N$ that generate $H$. We say $G$ contains a \emph{uniformly $N$-short} $H$, if for every finite generating set $S$ there exists a copy of $H$ in $G$ that is $N$-short with respect to $S$.
\end{defn}
 
Existence of uniformly $N$-short free subgroups or free semigroups in a group $G$ immediately gives the uniform bound $w(G) \geq \sqrt[N]{2}$ (see \cite[Proposition~2.4]{AlperinNoskov}).
This was used by Grigorchuk and de la Harpe \cite[Section~(A)]{GrigorchukHarpeGrowth} to show uniform exponential growth of torsion-free hyperbolic groups. 
Their work builds upon work of Gromov \cite[Theorem~5.3(E)]{Gromov} that was proved by Delzant \cite[Th\'{e}or\`{e}me~I]{DelzantTwoGen}.  

We say that a group $G$ has \emph{locally uniform exponential growth} if there is a constant $w_0 > 1$ such that every finitely generated subgroup $H$ of $G$ either has $w(H) \geq w_0$ or $H$ is virtually abelian.  
This property is sometimes called ``uniform uniform exponential growth''.
In the setting of groups acting properly on CAT(0) cube complexes, this is closely related to the strong Tits alternative of Sageev and Wise \cite[Theorem~1.1]{SageevWise} for a given subgroup $H$.  
Such bounds on the growth of subgroups have been shown by Mangahas for the mapping class group \cite{Mangahas}, depending on the complexity of the surface, and by Kar and Sageev for groups acting freely on square complexes \cite{KarSageev}.  
The bounds on exponential growth in this paper are also uniform over finitely generated subgroups with dependence only on the dimension of the cube complex.
We remark that bounds on uniform exponential growth need not pass to subgroups.  For example, the free product of two copies of Wilson's group has $w(G) \geq \sqrt[4]{2}$ coming from its action on its Bass-Serre tree.

The proof of Theorem~\ref{thm:2Dtorsion} relies on constructing a hyperbolic isometry with uniformly bounded word length from pairs of elliptic isometries of a 2-dimensional cube complex. Our construction also works for 3-dimensional cube complex. 

\begin{restatable*}{prop}{ellipticHyperbolic}\label{prop:elliptic hyperbolic}  Let $a$ and $b$ be a pair of isometries of a CAT(0) cube complex $X$ of dimension two or three. Then either 
\begin{enumerate}
 \item there exists a hyperbolic element in $\la a, b \ra$ whose length in $a,b$ is at most $L$, where $L$ is a constant that only depends on $\dim(X)$, or,
 \item $\langle a, b\rangle$ fixes a point in $X$.
\end{enumerate} 
If $X$ is 2-dimensional, then $L=12$.  
\end{restatable*}

We use the fact that the Burnside groups $B(2,2)$ and $B(2,6)$ are finite in the proof of \Cref{prop:elliptic hyperbolic}. 
However, this means that we cannot use our proof to obtain a similar result in higher dimensions. It remains open whether the above proposition holds in higher dimensions. 

\vspace{0.8cm}

In this article, we also impose extra conditions on the cube complex to show uniform exponential growth for groups acting on cube complexes of arbitrary dimension. 
In particular, we show the following (see Section~\ref{subsec:isolated flats} for definition of a CAT(0) cube complex with isolated flats).

\begin{restatable}{thmAlph}{ThmB}\label{thm:CAT(0)IFP}
There exists a constant $w_\dimX > 1$ depending only on $\dimX \in \N$ such that the following holds. 
Let $X$ be a CAT(0) cube complex of dimension $\dimX$ with isolated flats that admits a geometric group action. 
For any finitely generated group $G$ acting \emph{freely} on $X$, either  $w(G) \geq w_d$ or $G$ is virtually abelian.  
\end{restatable} 

Note that if the action of the group $G$ on $X$ is free and geometric then $G$ is relatively hyperbolic. Uniform exponential growth for relatively hyperbolic groups was proved by Xie \cite{Xie}.  However, we obtain bounds on growth depending only on dimension of $X$ rather than hyperbolicity constants.  Moreover, our result shows that $G$ has locally uniform exponential growth, which is new.

In the proof of \Cref{thm:CAT(0)IFP}, we use a variation (\Cref{lem:IFP dichotomy}) of work of Jankiewicz \cite{Jankiewicz} on cubical dimension of small cancellation groups (see Lemma~\ref{lem:Jankiewicz}). 
If instead of isolated flats, we impose the condition of hyperbolicity on our CAT(0) cube complex $X$, then we can use the same proof strategy as for Theorem~\ref{thm:CAT(0)IFP} to relax the requirement of a geometric group action to a weakly properly discontinuous (WPD) action.  

\begin{restatable*}{prop}{hyperbolic}\label{cor:hyperbolic} 
	There exists a constant $w_d > 1$ depending only on $d \in \N$ such that the following holds. 
	Let $X$ be a CAT(0) cube complex of dimension $\dimX$ that is \emph{hyperbolic}. For any finitely generated group $G$ acting \emph{freely} and \emph{weakly properly discontinuously} on $X$, either $w(G) \geq w_d$ or $G$ is virtually infinite-cyclic. 
\end{restatable*}

Similar results on uniform exponential growth for groups acting on hyperbolic spaces have been obtained before. 
In \cite[Theorem~13.1]{BreuillardFujiwara} Breuillard and Fujiwara show that if $G$ is a finitely generated group of isometries of a Gromov hyperbolic space, then either the exponential growth rate is bounded from below by a positive constant depending on the joint minimal displacement and the hyperbolicity constant or $G$ fixes a pair of points in the boundary. 
Before that, Besson, Courtois and Gallot showed that given $a>0, n \in \mathbb{N}$, there is a constant $c(n,a) >0$ such that if $M$ is a complete Riemannian manifold of dimension $n$ with pinched sectional curvature $\kappa_M \in [-a^2, -1]$ and $\Gamma$ is a finitely generated discrete group of isometries of $M$ then either $w(\Gamma) > e^{c(n,a)} >1$ or $\Gamma$ is virtually nilpotent \cite[Theorem~1.1]{BCG}.  The action of hyperbolic manifold groups and more generally $\delta$-hyperbolic groups is so nice that it is even possible to exhibit uniformly short free subgroups \cite[Theorem~1.1]{DeyKapovichLiu} \cite[Th\'{e}or\`{e}me~5.1]{Koubi}. 

A feature of our results is that they depend only on the dimension of the cube complex.  Our results may be useful in understanding the cubical dimension of finitely generated groups as in work of Jankiewicz \cite{Jankiewicz}.  

\vspace{0.8cm}

Besides groups that act properly on CAT(0) cube complexes, the results here allow us to give the first proofs of uniform exponential growth for several groups that admit improper actions on CAT(0) square complexes.  
We show that information about vertex stabilizers can be leveraged to give bounds on exponential growth.  

\begin{restatable*}{cor}{improperSquareCplx}\label{cor:improper2D}
	Suppose $G$ is a finitely generated group that acts by isometries on a CAT(0) square complex $X$ such that finitely generated subgroups of the vertex stabilizers are either virtually abelian or have uniform exponential growth bounded below by $w_0>1$. Then for any finitely generated subgroup $H \leq G$ one of the following holds:
	\begin{enumerate}
		\item $H$ has uniform exponential growth with $w(H) \geq \min\braces{\sqrt[600]{2}, w_0}$, or
		
		\item $H$ is virtually abelian, or
		
		\item $H$ stabilizes a flat or line in $X$.
	\end{enumerate}
\end{restatable*}

In particular, this lets us expand the list of acylindrically hyperbolic groups that are known to have locally uniform exponential growth.  In particular, we show that the Higman group and Artin groups with triangle-free defining graphs have locally uniform exponential growth (see \Cref{cor:Higman} and \Cref{thm:FCtypeArtin}).  

\subsection*{Organization}
In \Cref{sec:background}, we review background on cube complexes, CAT(0) spaces with isolated flats, past results related to building free semigroups and uniform exponential growth in cube complexes.  
In \Cref{sec:isolated flats}, we 
prove \Cref{thm:CAT(0)IFP} and \Cref{cor:hyperbolic}.  
In \Cref{sec:generating hyperbolic isometry}, we 
construct uniformly short hyperbolic isometries in any action on a cube complex of dimension at most 3 without global fixed point.  
This is the key tool needed to prove \Cref{thm:2Dtorsion} and \Cref{cor:Proper} in \Cref{sec:torsion2D}.  
We go on to prove locally uniform exponential growth of the Higman group and triangle-free Artin groups in  \Cref{sec:Applications} where we prove \Cref{cor:improper2D}.  

\subsection*{Acknowledgements}
The authors are grateful to Talia Fern\'os for helping to clarify the statement of Theorem~\ref{thm:2Dtorsion}
{color{blue}
	and the anonymous referee for helpful corrections.
} 
RG would like to thank Michah Sageev, Mark Hagen, and Giles Gardam for helpful conversations. She was partially supported by Israel  Science  Foundation Grant 1026/15 and EPSRC grant EP/R042187/1.
TN would like to thank Dave Futer, Sam Taylor, and Matthew Stover for their guidance and support on this project.  He would also like to thank Carolyn Abbott and Davide Spriano for bringing our attention to the Higman group.
TN was partially supported by NSF grant DMS--1907708, ISF grant 660/20, an AMS--Simons travel grant, and at the Technion by a Zuckerman Fellowship.
TN and KJ would like to thank the members of a virtual reading group--Yen Duong, Carolyn Abbott, Teddy Einstein, and Justin Lanier for helpful conversations on \cite{KarSageev}. 
The authors acknowledge support from U.S. National Science Foundation grants DMS 1107452, 1107263, 1107367 ``RNMS: Geometric Structures and Representation Varieties'' (the
GEAR Network).
The authors would also like to acknowledge the conference `Graphs, surfaces, and cube complexes' which was part of 2017-2018 Warwick EPSRC Symposium, where we started this collaboration.  
We also thank the anonymous referee for their helpful comments.

    \section{Background} \label{sec:background}

In this section we review fundamentals on cube complexes and past results on building free semigroups using CAT(0) cube complexes.
For more basics on CAT(0) cube complexes see Sageev's notes \cite{SageevPCMI}.

\subsection{Cube complexes and hyperplanes}
Let $X$ be a CAT(0) cube complex. Let $\Isom(X)$ denote the collection of cubical isometries of $X$. We will assume that any isometric group action on $X$ does not invert hyperplanes.  This is achieved by cubically subdividing $X$ once. 
We denote the fixed point set of a cubical isometry $a \in \Isom(X)$ by $\Fix(a) \subseteq X$. If two points $x,y$ are fixed by $a$, then the CAT(0) geodesic joining them is also fixed by $a$.  Therefore, $\Fix(a)$ is a connected and convex subspace of $X$ (with respect to the CAT(0) metric). 

Convex subcomplexes of CAT(0) cube complexes are particularly well-behaved.  
CAT(0) cube complexes are often regarded as high dimensional generalizations of trees because convex subcomplexes satisfy the \emph{Helly property}, that is, any collection of pairwise intersecting convex subcomplexes have nonempty intersection. 
 
A natural family of convex subspaces are hyperplanes.  
A \emph{hyperplane} is a subspace that separates the complex into two distinct half spaces by cutting every cube it intersects in half.  
We denote a hyperplane in $X$ by $\h$. 
Let $\Hyp(X)$ denote the collection of all hyperplanes of $X$. For a subcomplex $A$ of $X$, let $\Hyp(A)$ be the collection of hyperplanes of $X$ that separate a pair of points in $A$.
Note that when $A$ is convex with respect to the CAT(0) metric these are all the hyperplanes that intersect $A$.
A path joining two vertices of $X$ is called a \emph{combinatorial geodesic} if it is a path of minimum length in the 1-skeleton of $X$ joining the two points.  Note that every edge in a combinatorial geodesic uniquely corresponds to a hyperplane separating the vertices.

\begin{defn}[Cubical convex hull]
 Let $A$ be a subspace of $X$. The \emph{cubical convex hull} of $A$, denoted $\Hull(A)$, is the smallest convex subcomplex of $X$ containing $A$. 
\end{defn}

Hyperplanes give the cube complex a wallspace structure in the sense of Haglund and Paulin \cite{HaglundPaulin:Wallspace}.  
This wallspace structure produces a \emph{dual cube complex} from a collection of hyperplanes of $X$ using Sageev's construction \cite{Sageev:EndsNPC} (see \cite[Lecture~2]{SageevPCMI} for more details).
Since hyperplanes encode combinatorial geodesics in a cube complex, the hyperplanes that cross a subcomplex determine its cubical convex hull. We record this observation.  

\begin{obs}
	\label{obs:cHull}
	Let $A$ be a subcomplex of $X$. Then the subcomplex $\Hull(A)$ is isomorphic to the cube complex dual to $\Hyp(A)$.  
\end{obs}

We say that two subcomplexes $A,B \subset X$ are \emph{parallel} when there exists $p \geq 0$ such that $A \times [0,p]$ embeds isometrically in $X$ such that $A \times \braces{0} = A$ and $A \times \braces{p} = B$.  
In light of \Cref{obs:cHull}, hyperplanes can be used to detect when two convex subcomplexes are parallel.

\begin{lem}[{\cite[Lemma~2.8]{Huang:QIraags}, \cite[Lemma~2.7]{HJP:cube2DArtin}}]
	Two convex subcomplexes are parallel when they are dual to the same hyperplanes.  Moreover, they are equal if and only if there are no hyperplanes separating them.
\end{lem}

\subsection{Isometries of CAT(0) cube complexes and their associated subcomplexes}
Let $X$ be a finite dimensional CAT(0) cube complex.
Given a hyperplane $\h$ of $X$ and $g$ a hyperbolic isometry of $X$, we say $g$ \emph{skewers} $\h$ if for some choice of halfspace $\hs$ associated to $\h$ we have, $g^p \hs \subset \hs$ for some $p \geq 1$.  Such $p$ can be taken to be at most $\dim(X)$.  Equivalently, $g$ skewers $\h$ if any CAT(0) axis of $g$ intersects $\h$ in exactly one point.  We say $g$ is \emph{parallel} to $\h$ if any axis for $g$ is contained in the  $R$-neighborhood of $\h$ for some $R \geq 0$. 
A hyperbolic isometry $g$ is \emph{peripheral} to $\h$ if it neither skewers it nor is parallel to it.

Given a hyperbolic isometry $g$ of $X$, the \emph{skewer set} of $g$, denoted $\sk(g)$, is the collection of all hyperplanes skewered by $g$. 
The \emph{parallel subcomplex} $Y_g$ of a hyperbolic isometry $g$ is the maximal subcomplex of $X$ contained in the intersection of peripheral halfspaces containing the axes of $g$ (see~\cite{KarSageev}).  
Let $\mathrm{p}(g)$ be the collection of hyperplanes parallel to $g$.
The parallel subcomplex $Y_g$ is a $\abrackets{g}$-invariant subcomplex of $X$ dual to 
$$ \Hyp(Y_g) = \sk(g) \cup \mathrm{p}(g). $$
This subcomplex naturally decomposes as a product $Y_g = E_g\times K_g$ where $E_g$ is dual to $\sk(g)$ and $K_g$ is dual to $\mathrm{p}(g)$ because any hyperplane in $\mathrm{p}(g)$ crosses every hyperplane of $\sk(g)$.
The complex $E_g$ is a \emph{Euclidean subcomplex}, that is, 
when equipped with the combinatorial metric,
it isometrically embeds in Euclidean space \cite[Lemma~2.4]{Jankiewicz}. 
The action of $g$ on $Y_g$ respects the decomposition where $g$ acts as a translation on $E_g$, and has a fixed point in its action on $K_g$.
Moreover, for every axis $\ell_g$ of $g$, the subcomplex $Y_g$ contains $\Hull(\ell_g)$, which is isometric to $E_g$.

\subsection{Rank one isometries}
We record some observations about rank one isometries of CAT(0) cube complexes. 
Let $X$ be a finite dimensional locally finite CAT(0) cube complex.  
An isometry $g$ of $X$ is called \emph{rank one} if it is hyperbolic and an axis of $g$ does not bound a half flat. A bi-infinite geodesic in $X$ is called a \emph{rank one geodesic} if it does not bound a half flat.  
The \emph{visual boundary} $\boundary X$ of a geodesic metric space $X$ consists of equivalence classes of geodesic rays emanating from a base point where two rays are said to be equivalent if they have finite Hausdorff distance.  
The homeomorphism type of the boundary does not depend on the choice of base point.

\begin{lem}\label{lem:Limit set of rank one isometry}\cite[Lemma III.3.3]{BallmannLectures}
 Let $g$ be a rank one isometry of a locally compact CAT(0) space $Y$. Then $g$ fixes exactly two points in the visual boundary $\boundary Y$, denoted $\Lambda(g)$.  
\end{lem}

We now show that $g$ acts cocompactly on its parallel subcomplex $Y_g$ when $g$ is a rank one isometry.

	\begin{prop}\label{prop:Yg maximal}
		Let $g$ be a hyperbolic isometry of $X$, and let $\ell_g$ be a CAT(0) axis of $g$. Then any bi-infinite CAT(0) geodesic within finite Hausdorff distance from $\ell_g$ is contained in the parallel subcomplex $Y_g$ of $g$.
		
		Moreover, if $g$ is rank one, then the parallel subcomplex $Y_g$ is the maximal convex subspace of $X$ with $\boundary Y_g = \Lambda(g)$.  Furthermore, $Y_g$ is a quasi-line on which $g$ acts cocompactly.
	\end{prop}
	\begin{proof} 
		Let $\ell$ be a bi-infinite geodesic such that $d_{Haus}(\ell, \ell_g) <\infty$. 
		It is straightforward to see that the two geodesics are asymptotic. 
		By the Flat strip theorem \cite[Theorem~II.2.13]{BridsonHaefliger}, we get that $\ell$ and $\ell_g$ are in fact parallel. Therefore $\ell \subset Y_g$. 
		
		Suppose $g$ is rank one. The limit set $\Lambda(g)$ has two points by \Cref{lem:Limit set of rank one isometry}. 
		Any bi-infinite geodesic joining $\Lambda(g)$ has finite Hausdorff distance from $\ell_g$, 
		and by the above is contained in $Y_g$. 
		Furthermore, by \cite[Lemma~4.8]{Hagen2020} $\Hull(\ell_g)$ is contained in a bounded neighborhood of the axis $\ell_g$, so the action of $g$ on $\Hull(\ell_g)$ is cocompact.
		The parallel subcomplex $Y_g$ decomposes as a product of cube complexes $E_g\times K_g$,
		where $E_g$ is isometric to $\Hull(\ell_g)$
		and $K_g$ is bounded since $g$ is rank one. 
		Thus $\boundary Y_g = \boundary E_g = \Lambda(g)$. 
		The action of $g$ preserves the product decomposition and acts as a translation on $E_g$.
		Hence, it follows that the action of $g$ on $Y_g$ is cocompact.
	\end{proof}

We record the following consequence that will be used in the proof of \Cref{thm:CAT(0)IFP}.

\begin{lem}\label{lem:Ya equal to Yb}
 Let $a$ and $b$ be rank one isometries of $X$ such that $\Lambda(a) = \Lambda(b)$. Then $Y_a = Y_b$. 
\end{lem}
\begin{proof}
Since $\Lambda(a) = \Lambda(b)$, an axis of $a$ and an axis of $b$ are finite Hausdorff distance apart. Then by \Cref{prop:Yg maximal}, $\ell_b \subset Y_a$ and $\ell_a \subset Y_b$.  
By Proposition~\ref{prop:Yg maximal}, $Y_a$ is the maximal convex subspace of $X$ with $\boundary Y_a = \Lambda(a)$. Therefore $Y_b \subseteq Y_a$ and vice versa. Thus they are equal.  
\end{proof}

\subsection{Types of group actions and Bieberbach's Theorem}

We recall the definitions of free, proper, and discrete actions on possibly locally infinite CW complexes. 
Let $X$ be a CW complex and let $G$ be a discrete group acting cellularly by isometries on $X$. 
The action of $G$ is \emph{proper} when for every compact set $K$ in  $X$, the collection $\{g \in G | gK \cap K \neq \emptyset\}$ is finite.  Since $X$ is a CW complex, this is equivalent to requiring that every cell stabilizer is finite (see for instance, \cite{KapovichMisha}).

The action of $G$ on $X$ is \emph{free} if every point in $X$ has trivial stabilizer. 
We say the action is \emph{discrete} if $G$ is a discrete subgroup of homeomorphisms of $X$ with respect to the compact open topology. 
For CW complexes, proper actions are always discrete. 
We recall the following version of Bieberbach's theorem for later reference. 

\begin{thm}[Bieberbach's Theorem]\cite[Corollary 4.1.13]{ThurstonBook}\label{thm:Bieberbach}
	For each dimension $n$, there is an integer $m$ such that any group acting discretely by isometries on Euclidean $n$-space $\mathbb{E}^n$ has an abelian subgroup of index at most $m$. 
\end{thm}

\subsection{Flats in $\CAT(0)$ spaces} \label{subsec:isolated flats}

 \begin{defn} \label{def:flat}
 	Let $X$ be a $\CAT(0)$ space. For $k \geq 2$, a \emph{($k$-)flat} in $X$ is an isometrically embedded copy of Euclidean space $\mathbb{E}^k$.
 	A \emph{half-flat} in $X$ is an isometrically embedded copy of $\R \times \R_+$.
 \end{defn}

 The reader is referred to \cite{Hruska:GeometricInvariants} for details on CAT(0) spaces with isolated flats. We recall the definition here.  

  \begin{defn} \label{defn:CAT(0)IFP}
   A $\CAT(0)$ space $X$ has {\it isolated flats} if there is a non-empty $\Isom(X)$-invariant collection of flats $\F$, of dimension at least two, such that the following conditions hold:
   \begin{enumerate}
    \item {(Maximal)} There exists a constant $D < \infty$ such that each flat in $X$ lies in the $D$-tubular neighborhood of some $F \in \F$.
    \item {(Isolated)} For every $\rho < \infty$ there exists $\kappa(\rho)<\infty$ such that for any two distinct flats $F,F' \in \F$, $\diam(N_{\rho}(F) \cap N_{\rho}(F')) <\kappa(\rho)$. 
   \end{enumerate}
  \end{defn}

  We say a CAT(0) cube complex $X$ has isolated flats if $X$ with its CAT(0) metric is a CAT(0) space with isolated flats. 
  CAT(0) cube complexes are particularly well-adapted to studying isolated flats because hyperplanes inherit the isolated flats property.  
  
  \begin{lem}\label{lem:hyperplane} 
  	Let $X$ be a CAT(0) cube complex with isolated flats. Let $\h$ be a hyperplane of $X$. Then either $\h$ does not have any flats or it is also a CAT(0) cube complex with isolated flats. 
  \end{lem}
  \begin{proof} 
  	Let $F$ be a flat contained in $\h$. Every flat in $\h$ is a flat in $X$ and hence there exists a maximal flat $F'$ in $\F$ containing $F$. Since intersection of two convex sets is convex in a CAT(0) space, $\h \cap F'$ is convex. Thus $\h \cap \F$ is the non-empty collection of maximal flats in $\h$ that satisfy Definition~\ref{defn:CAT(0)IFP}.
  \end{proof}
  
  \begin{lem}\label{lem:half flat close to flat}
  	Any half flat in a CAT(0) space with isolated flats that admits a geometric group action is contained in a bounded Hausdorff neighborhood of a maximal flat in $\F$. 
  \end{lem}
  \begin{proof}
  	This follows because such spaces are relatively hyperbolic with respect to the collection of maximal flats by Hruska and Kleiner \cite{HruskaKleiner}.
  \end{proof}

  \begin{lem}\label{lem:b stabilizes end points of a}
  	Let $a$ be a rank one isometry and $b$ be a hyperbolic isometry of a CAT(0) space $X$ with isolated flats that admits a geometric group action. If $b$ fixes $\Lambda(a)$ then $b$ is also a rank one isometry and $\Lambda(a) = \Lambda(b)$. 
  \end{lem}
  \begin{proof}
  	Let $\overline{b}$ be the homeomorphism of $\partial X$ induced by $b$. By \cite[Theorem 3.3]{Ruane:Dynamics}, the set of elements of $\partial X$ fixed by $\overline{b}$ is equal to $\partial \text{Min}(b)$. Therefore, $\Lambda(a) \subseteq \partial \text{Min}(b)$. If $b$ is not rank one then any axis of $b$ bounds a half-flat, which is contained is a bounded neighborhood of a maximal flat $F$ by \Cref{lem:half flat close to flat}.
  	Hence, $\partial \Min(b) \subseteq \partial F$. 
  	However, $\Lambda(a)$ cannot be contained in the boundary of a maximal flat because $a$ is rank one.  
  	Thus, $b$ is also a rank one isometry and $\Lambda(a) = \Lambda(b)$ by Lemma~\ref{lem:Limit set of rank one isometry}. 
  \end{proof}

\subsection{Past results building free semigroups in cube complexes}
Uniform exponential growth is typically proved by generating uniform length free semigroups. For instance, for a 2-dimensional cube complex Kar and Sageev show the following: 

\begin{prop}\cite[Proposition 15]{KarSageev} 
	\label{prop:KS 2 element dichotomy} 
Let $a, b$ be two distinct hyperbolic isometries of a CAT(0) square complex $X$. Then either 
\begin{enumerate}
\item $\abrackets{a,b}$ contains a 10-short free semigroup, or 
\item there exists a Euclidean subcomplex of $X$ invariant under $\abrackets{a,b}$.  
\end{enumerate}
\end{prop}


In \cite{Jankiewicz}, Jankiewicz obtains the following generalization of Proposition~\ref{prop:KS 2 element dichotomy} to higher dimensional cube complexes. 

\begin{lem} \cite[Lemma 4.2]{Jankiewicz}\label{lem:Jankiewicz}
Let $a,b$ be two distinct hyperbolic isometries of a $\dimX$-dimensional CAT(0) cube complex $X$ such that $\abrackets{a,b}$ acts freely on $X$. Then one of the following hold
\begin{enumerate}
\item {(Short free semigroup)} there exists a constant $L = L(\dimX) <\infty$ such that $\abrackets{a,b}$ contains an $L$-short free semigroup, or, 
\item {(Stabilize hyperplane)} one of $\abrackets{ b^N, a^{-\dimX!} b^N a^{\dimX!} }$ or $\abrackets{ a^N, b^{-\dimX!} a^N b^{\dimX!} }$ stabilizes a hyperplane of $X$, or 
\item {(Virtually abelian powers)} 
the group  $\abrackets{ a^N, b^N }$ is virtually abelian.
\end{enumerate}
where $N = \dimX!K_3!$ and $K_3$  is the Ramsey number $R(\dimX +1, 3)$.
\end{lem}

Compared to Proposition~\ref{prop:KS 2 element dichotomy}, Lemma~\ref{lem:Jankiewicz} requires taking powers. 
However, under the additional assumption that $X$ has isolated flats and a geometric group action (as in Section~\ref{sec:isolated flats}), we can recover Proposition~\ref{prop:KS 2 element dichotomy} in all dimensions (see \Cref{lem:IFP dichotomy}).    

\section{CAT(0) spaces with isolated flats}\label{sec:isolated flats}
The goal of this section is to prove Theorem~\ref{thm:CAT(0)IFP}. We start by proving some results that lead up to \Cref{lem:IFP dichotomy}, which is the main lemma used to prove Theorem~\ref{thm:CAT(0)IFP}. In what follows, let $(X, \F)$ be a CAT(0) space with isolated flats that admits a geometric group action.  
When $X$ is a CAT(0) cube complex that admits a geometric group action, $X$ is finite dimensional and locally finite. 

\begin{lem}
	\label{lem:TwoElementPower} 
	Let $a$ and $b$ be a pair of hyperbolic isometries of $X$ such that $\la a^N, b^M \ra \leq \Isom(X)$ stabilizes a flat for some $N, M \in \mathbb{Z} \setminus \{0\}$. Then $\la a, b \ra \leq \Isom(X)$ stabilizes a maximal flat in $X$. 
\end{lem}
\begin{proof}
Let $\la a^N, b^M \ra$ stabilize a flat $F_0$. By \Cref{defn:CAT(0)IFP}{(Maximal)}, there exists a maximal flat $F \in \F$ such that $F_0$ is contained in a $D$-neighborhood of $F$. Let $\ell_a$  and $\ell_b$ be axes of $a$ and $b$ respectively. 
Then $\ell_a$ and $\ell_b$ are contained in a bounded neighborhood of $F$. 
This is because $a^N$ has an axis in $F_0$, $\ell_a$ is also an axis of $a^N$, and any two axes of $a^N$ are parallel in $X$. The same is true for $\ell_b$ and $b^M$. The axis $\ell_a$ (resp. $\ell_b$) is also in a bounded neighborhood of $aF$ (resp. $bF$). Since the collection $\F$ is $\Isom(X)$-invariant, $aF, bF \in \F$. Now by Definition~\ref{defn:CAT(0)IFP}{(Isolated)} and maximality of $F$, we get that $F = aF$ and $F=bF$. Thus $a$ and $b$ stabilize $F$. 
\end{proof}

\begin{lem}
	\label{lem:TwoElementPowerCyclic} 
	Let $a$ and $b$ be a pair of hyperbolic isometries of $X$ such that $\la a^N, b^M \ra \leq \Isom(X)$ stabilizes a line for some $N, M \in \mathbb{Z} \setminus \{0\}$. Then either $\la a, b \ra$ stabilizes a flat or $a$ and $b$ are rank one isometries and $\la a, b \ra$ fixes $\Lambda(a)=\Lambda(b) \subset \partial X$.
\end{lem}
\begin{proof}
	Let $\ell$ be a line stabilized by $\la a^N, b^M \ra$. Suppose $\ell$ is contained in a tubular neighborhood of some maximal flat $F\in \F$. 
	Since maximal flats in $X$ are isolated, $\la a^N, b^M \ra$ stabilizes $F$. Now by Lemma~\ref{lem:TwoElementPower}, $\la a, b \ra$ stabilizes a flat. 
	
	Now suppose $\ell$ is not contained in a tubular neighborhood of any maximal flat $F\in \F$. Then by \Cref{lem:half flat close to flat}, $\ell$ is a rank one geodesic. 
	Hence, $a^N$ and $b^M$ are rank one isometries, and so are $a$ and $b$. 
	The axes of $a$ and $b$ fellow travel so they share the same endpoints $\Lambda = \Lambda(a) = \Lambda(b)$.  
	Thus $\la a, b \ra$ fixes $\Lambda$ by Lemma~\ref{lem:Limit set of rank one isometry}.
\end{proof}

We obtain a similar result when an element and its conjugate stabilize a flat or line.  

\begin{lem}
	\label{lem:TwoElementConjugate} 
	Let $a$ and $b$ be a pair of hyperbolic isometries of $X$ such that $\la a, bab^{-1} \ra \leq \Isom(X)$ stabilizes a flat or a line. Then either $\la a, b \ra$ stabilizes a flat or $a$ and $b$ are rank one isometries and $\la a, b \ra$ fixes $\Lambda(a)=\Lambda(b) \subset \partial X$.  
\end{lem} 

\begin{proof}
Let $E$ be a line or flat stabilized by $ \abrackets{a, bab^{-1}}$ and suppose $E$ is contained in a tubular neighborhood of some maximal flat $F\in \F$.  
Let $\ell_a$  and $\ell_{bab^{-1}} := b\ell_a$ be fixed axes of $a$ and $bab^{-1}$ respectively. 
Then $\ell_a$ and $b\ell_a$ are contained in a bounded neighborhood of $F$. The axis $\ell_a$ (resp. $b\ell_a$) is also in a bounded neighborhood of $aF$ (resp. $bF$). Since the collection $\F$ is $\Isom(X)$-invariant, $aF, bF \in \F$. 
Now by \Cref{defn:CAT(0)IFP}({Isolated}), we get that $F = aF$ and $F=bF$. 
Therefore $a$ and $b$ stabilize $F$.

Now suppose $E$ does not lie in a tubular neighborhood of a maximal flat. It follows that $\dim(E) = 1$, so $E$ is an axis for $a$. 
By \Cref{lem:half flat close to flat}, $a$ is a rank one isometry. Both $a$ and $bab^{-1}$ stabilize $\Lambda(a)$, which is equal to the pair of end points of $E$. Therefore $b$ and hence $\la a, b \ra$ also fixes $\Lambda(a)$. Then by \Cref{lem:b stabilizes end points of a}, we conclude that $b$ is also a rank one isometry and $\Lambda(a) = \Lambda(b)$. 
\end{proof}
   
In light of the above results, we are able to upgrade \Cref{lem:Jankiewicz} in the setting of isolated flats.
For the remainder of this section let $X$ be a CAT(0) cube complex of dimension $\dimX$ with isolated flats that admits a geometric group action.

\begin{prop}\label{lem:IFP dichotomy} 
Let $a, b \in \Isom(X)$ be a pair of hyperbolic isometries of $X$ such that $\abrackets{a,b}$ acts freely on $X$. There exists a constant $M = M(\dimX) < \infty$ such that either \begin{enumerate}
	\item $\abrackets{a,b}$ contains an $M$-short free semigroup, or
	\item the subgroup $\abrackets{a, b}$ stabilizes a flat, or
	\item $a$ and $b$ are rank one isometries and $\la a, b \ra$ fixes $\Lambda(a) = \Lambda(b) \subset \partial X$.  
\end{enumerate}
\end{prop} 
\begin{proof}
	The proof is by induction on the dimension of $X$.  
	For the base cases, if $\dim(X) = 1$ then we are done by \Cref{lem:treeFreeSemigroup} with $M(1) = 4$.
	If $\dim(X) = 2$ then the conclusions are satisfied by \cite[Main Theorem]{KarSageev},  with $M(2) = 10$. 
	
	For the induction step, apply \Cref{lem:Jankiewicz}.  Let $M(d) = \max\braces{ (2\dimX!+N)\cdot M(d-1), L(d) }$ where $L(d)$ is the constant from \Cref{lem:Jankiewicz}.  
	If the condition {(Short free semigroup)} is satisfied then we are done.  
	If the condition {(Virtually abelian powers)} is satisfied then $\abrackets{a^N,b^N}$ stabilizes a flat or line by the Flat Torus Theorem \cite{BridsonHaefliger}. 
	By \Cref{lem:TwoElementPower} or \Cref{lem:TwoElementPowerCyclic}, conclusion (2) or (3) holds.
	If the condition {(Stabilize hyperplane)} is satisfied then up to switching $a$ and $b$, the group $H = \abrackets{a^N, b^{-\dimX!}a^Nb^{\dimX!}  }$ stabilizes a hyperplane.  
	Since hyperplanes are convex in the cube complex, the elements $a^N$ and $b^{-\dimX!}a^Nb^{\dimX!}$ have axes in $\h$ \cite[Chapter~II.6~Proposition~6.2(4)]{BridsonHaefliger}. Hence, they are both hyperbolic in the restricted action of $H$ on $\h$.  Moreover, the restricted action is free.
	Therefore by the induction hypothesis, $\abrackets{a^N,b^{-\dimX!}a^Nb^{\dimX!}}$ either contains an $M(d-1)$-short free semigroup or 
	$\abrackets{a^N,b^{-\dimX!}a^Nb^{\dimX!}}$ stabilizes a flat in $\h$ or $a^N$ and $b^{-\dimX!}a^Nb^{\dimX!}$ are rank one isometries of $\h$ fixing $\Lambda(a^N) = \Lambda(b^{-\dimX!}a^Nb^{\dimX!}) \subset  \partial \h \subset \partial X$.
	In the case when $\abrackets{a^N,b^{-\dimX!}a^Nb^{\dimX!}}$ stabilizes a flat in $\h$, conclusion (2) or (3) holds by \Cref{lem:TwoElementConjugate} followed by \Cref{lem:TwoElementPower} or \Cref{lem:TwoElementPowerCyclic}.
	In the remaining case, $b^{-\dimX!}$ fixes $\Lambda(a^N)$ because $b^{-\dimX!}\Lambda(a^N) = \Lambda(b^{-\dimX!}a^Nb^{\dimX!})$.  By \Cref{lem:b stabilizes end points of a}, $b^{-\dimX!}$ is a rank one isometry  
	of $\h$
	with $\Lambda := \Lambda(b^{-\dimX!}) = \Lambda(a^N)$.  
	
	If $a$ and $b$ are also rank one in $X$ then
	$\abrackets{a,b}$ fixes $\Lambda = \Lambda(a) = \Lambda(b)$ and we are done.
	Suppose now that $b$ is not rank one in $X$.  
	It must be that the subgroup $\la a, b \ra$ stabilizes a maximal flat $F \in \F$.
	Indeed, by \Cref{lem:b stabilizes end points of a} the element $a$ is also not rank one in $X$.  By \Cref{lem:half flat close to flat} any axis of $a$ (respectively $b$) is within finite Hausdorff distance of a maximal flat $F_a \in \F$ (respectively $F_b \in \F$) such that $a F_a = F_a$ (respectively $b F_b = F_b$).  
	Since $\Lambda(a) = \Lambda(b)$ the maximal flats $F_a, F_b$ are equal.
	Therefore $F := F_a = F_b$ is stabilized by $\la a, b \ra$.
\end{proof}

Now that we understand how pairs of hyperbolic isometries stabilize flats and pairs of points in $\partial X$, we are able to extend this to any finite collection.

\begin{lem}\label{lem:FiniteSetAbelian} 
	Let $s_1, \ldots, s_n$ be a collection of hyperbolic isometries of $X$ such that for each $1 \leq i\neq  j \leq n$, either $\la s_i, s_j \ra$ stabilizes a flat or $s_i, s_j$ are rank one isometries and $\la s_i, s_j \ra$ stabilizes $\Lambda(s_i) = \Lambda(s_j)$. Then either $\la s_1, \ldots, s_n \ra$ stabilizes a flat or for every $i \geq 1$, $s_i$ is a rank one isometry and $\la s_1, \ldots, s_n \ra$ stabilizes $\Lambda(s_1) = \cdots = \Lambda(s_n) \subset \partial X$. 
\end{lem}

\begin{proof}  
	If $s_1$ is a rank one isometry then so is every $s_i$. Indeed, rank one isometries do not stabilize any flat and every hyperbolic isometry $g \in \stab(\Lambda(s_1))$ is rank one with $\Lambda(g) = \Lambda(s_1)$.
	
	We may thus assume that every pair in $\{s_1, \ldots, s_n\}$ stabilizes a flat. 
	Let $E_i$ be a flat stabilized by $\la s_1, s_i \ra$ where $2 \leq i \leq n$. 
	By \Cref{defn:CAT(0)IFP}{(Maximal)}, $E_i$ is contained in a tubular neighborhood of  maximal flat  $F_i \in \F$. 
	We argue as in proof of Lemma~\ref{lem:TwoElementPower} that $\la s_1, s_i \ra$ stabilizes $F_i$. 
	Let $\ell$ be an axis of $s_1$. 
	Then $\ell$ is contained in a bounded neighborhood of $F_{i}$ for all $2 \leq i \leq n$. Therefore by \Cref{defn:CAT(0)IFP}{(Isolated)}, $F_{i} = F_{j} =: F$ for all $i \neq j$. 
	Therefore, each $s_i$ stabilizes $F$ and hence $\la s_1, \ldots, s_n \ra$ stabilizes a flat in $X$. 
\end{proof}

We are now ready to prove \Cref{thm:CAT(0)IFP}.

\ThmB*

\begin{proof}
	Let $S = \{s_1, \ldots, s_n\}$ be a finite generating set for $G$. Since the action is free, each $s_i$ is a hyperbolic isometry of $X$. For every $1 \leq i \neq j \leq n$, consider the pair $s_i$ and $s_j$. By \Cref{lem:IFP dichotomy} applied to $s_i$ and $s_j$, if  there exists a constant $M = M(\dimX) <\infty$ such that $\abrackets{s_i, s_j}$ contains an $M$-short free semigroup,  then $w(G) \geq \sqrt[M]{2} =: w_d $. 
	So suppose for all pairs $i \neq j$, either $\la s_i, s_j \ra$ stabilizes a flat or $s_i, s_j$ are rank one isometries and $\la s_i, s_j \ra$ stabilizes $\Lambda(s_i) = \Lambda(s_j)$. 
	By \Cref{lem:FiniteSetAbelian}, either $G = \la s_1, \ldots, s_n \ra$ stabilizes a flat or for every $i \geq 1$, $s_i$ is a rank one isometry and $G$ stabilizes $\Lambda := \Lambda(s_1) = \cdots = \Lambda(s_n) \subset \partial X$. 
	
	First suppose $G$ stabilizes a flat. Since $G$ acts freely on $X$, it is a discrete subgroup of isometries of the flat. Thus by Bieberbach's theorem, the group $G$ is virtually abelian. 
	 
	Now suppose each $s_i$ is a rank one isometry and $G$ stabilizes $\Lambda$. By \Cref{lem:Ya equal to Yb}, the parallel subcomplex $Y_{s_i}$ coincides with $Y_{s_j}$ for all $i \neq j$. Let $Y := Y_{s_1} = \dots =  Y_{s_n}$.
	By \Cref{prop:Yg maximal}, $Y$ is a quasi-line. 
	Since $G$ acts freely on $X$, $G$ acts freely and hence properly on $Y$. 
	By \Cref{prop:Yg maximal}, each $s_i$ acts cocompactly on $Y$.
	Hence, $G$ acts geometrically on $Y$. 
	Thus, $G$ is two-ended and hence virtually infinite cyclic.  
\end{proof}

\vspace{.4cm}
If instead of isolated flats, we impose the condition of hyperbolicity on our CAT(0) cube complex $X$, then we can use the same proof strategy as for Theorem~\ref{thm:CAT(0)IFP} to get uniform exponential growth where the requirement of a geometric group action is replaced with the weaker hypothesis of a weakly properly discontinuous action introduced Bestvina and Fujiwara \cite[Section~3]{BF}.
(see \cite{Osin} for more details on acylindrical hyperbolicity).
We first recall a lemma from a paper of Dahmani, Guirardel, and Osin. 

\begin{lem}\cite[Lemma 6.5]{DahmaniGuirardelOsin}
	\label{lem:DGO}
	Let $G$ be a group acting on a $\delta$-hyperbolic space $X$ and let $h \in G$ be a hyperbolic WPD element. Then $h$ is contained in a unique maximal virtually cyclic subgroup of $G$, $E(h)$. Moreover, $E(h) = \{g \in G | d_{Haus}(g\ell, \ell) < \infty\}$, where $\ell$ is a quasi-geodesic axis of $h$ in $X$.  
\end{lem}

In more general actions on hyperbolic spaces, stabilizers of endpoints of hyperbolic isometries need not be virtually abelian let alone virtually cyclic.  By requiring that hyperbolic isometries be WPD, we can prove the following.  

\hyperbolic

\begin{proof} 
Let $a, b \in G$ be two distinct hyperbolic isometries of $X$. Suppose $\la a^M, b^N \ra$ stabilizes a line $\ell$ in $X$. Then $a^M, b^N$ fix the endpoints of $\ell$. Since $a$ and $a^N$ share the same fixed points in $\partial X$, we conclude that $a$ and $b$ fix the same pair of points in $\partial X$. This implies that $b\ell_a$ has finite Hausdorff distance from $\ell_a$. Thus by Lemma~\ref{lem:DGO}, $b \in E(a)$ and $\la a,b \ra$ is virtually cyclic. 
Similarly, if $a$ and $bab^{-1}$ stabilize a line then $\abrackets{a,b}$ is virtually cyclic because $b \in E(a)$. 

Now let $\{s_1, \ldots, s_n \}$ be a finite generating set of $G$ such that each $s_i$ acts as a hyperbolic isometry of $X$.   
Then the result follows as in the proof of \Cref{lem:IFP dichotomy}, \Cref{lem:FiniteSetAbelian} and \Cref{thm:CAT(0)IFP}.  
\end{proof}

\begin{remark}
Let $G$ be a group acting freely on a CAT(0) cube complex $X$. If $G$ satisfies the conclusions of \Cref{lem:TwoElementPower}, \Cref{lem:TwoElementPowerCyclic}, \Cref{lem:TwoElementConjugate}, and \Cref{lem:FiniteSetAbelian}, then we can use \Cref{lem:Jankiewicz} as above to obtain a uniform exponential growth result for $G$. 

A simple example to illustrate the difficulty of working with powers is the following: Let $G = \la x, y | [x^2, y^2]=1 \ra$. Then $G$ acts freely on a CAT(0) cube complex because its presentation 2-complex can be subdivided into a CAT(0) square complex. The subgroup $\la x^2, y^2 \ra$ is free abelian. However, $G$ is far from being virtually abelian. The group $G$ is isomorphic to a (f.g.\ free)-by-cyclic group. A brief outline is as follows: we change the presentation by replacing the generator $y$ by $z=xy$ and we get $G=\la x, z | x^2zx^{-1} z x^{-2} z^{-1}xz^{-1} = 1 \ra$. Consider the epimorphism $\phi\colon G \to \mathbb{Z}$ which sends $x$ to $1$ and $z$ to 0. Then by Brown's criterion \cite{Brown} (see also \cite[Section 5]{DunfieldThurston}), $\text{ker}(\phi)$ is finitely generated and by Magnus's Freiheitssatz \cite{Magnus} it is a free group. Thus $G$ is isomorphic to $\text{ker}(\phi)\rtimes\mathbb{Z}$. 
\end{remark}

    \section{Generating hyperbolic isometry}\label{sec:generating hyperbolic isometry}
In this section, our goal is to produce a hyperbolic isometry $g$ of a CAT(0) cube complex $X$ by composing two elliptic isometries $a$ and $b$, such that $g$ has uniform length in $a$ and $b$. 
For example, if $a$ and $b$ are two elliptic isometries of a tree with disjoint fixed points, then $g= ab$ is a hyperbolic isometry of the tree \cite[I.~Proposition~26]{Serre}.  
We obtain an analogous result when $X$ has dimension 2 or 3.

We first prove a general lemma quantifying how elliptic isometries of CAT(0) cube complexes interact with the cubical hull of their fixed sets. 

\begin{lem}\label{lem:cHull fixed}
Let $X$ be a $\dimX$-dimensional CAT(0) cube complex and let $a$ be an elliptic isometry of $X$. Then $\Hull(\Fix(a))$ is pointwise fixed by $a^{k}$ where $k = \text{lcm}\{1, 2, \ldots, \dimX \}$. 
\end{lem}
\begin{proof}
Let $A:= \Fix(a)$. By \Cref{obs:cHull}, $\Hull(A)$ is dual to the collection of hyperplanes $\Hyp(A)$. We will show that $a^k$ fixes each hyperplane in $\Hyp(A)$, which implies that $a^k$ pointwise fixes $\Hull(A)$.

Let $A_0$ be the union of all open cubes of $X$ that intersect $A$ non-trivially. 
Then $a$ preserves every cube $C \in A_0$. Also every hyperplane that crosses $C$ is in $\Hyp(A)$. Since $C$ contains fixed points of $a$, the action of $a$ on $C$ is determined by a permutation of at most $\dimX$ hyperplanes. 
However, the order of every cyclic permutation group on $\dimX$ objects divides $k = \lcm\braces{1, 2, \dots, \dimX}$, so $a^k$ is a trivial permutation of hyperplanes crossing a cube. Therefore, $a^k$ fixes each hyperplane in $\Hyp(A)$. 
\end{proof}

The next lemma differs from \Cref{prop:elliptic hyperbolic} in that the conclusion involves the subgroup $\la a^k, b^k \ra$ instead of $\la a, b \ra$. 
\begin{lem}\label{lem:separatedFixSetHyperbolic}
	Let $a,b$ be two elliptic isometries of a CAT(0) cube complex $X$ of dimension $\dimX$. If their fixed sets are separated by a hyperplane then there exists a hyperbolic isometry $g \in \la a, b \ra$, such that $g$ has length at most $2\dimX$ in $a$ and $b$. 
	Otherwise, the subgroup $\la a^{k}, b^k \ra$, where 
	$k=\text{lcm}\{1, 2, \ldots, \dimX \}$, fixes pointwise the intersection $\Hull(\Fix(a)) \cap \Hull(\Fix(b))$. 
\end{lem}
\begin{proof}
Let $\h$ be a hyperplane that separates $A:=\Fix(a)$ and $B:= \Fix(b)$. Since $X$ is $\dimX$-dimensional, there are two hyperplanes in $\{\h, a\h, a^2\h, \ldots, a^d \h\}$ that are either equal or disjoint. If $\h = a\h$, then the point in $\h$ closest to $A$ is fixed by $a$, which is not possible since $\h$ separates $A$ and $B$. Suppose $\h = a^k\h$ for some $2 \leq k \leq d$, and for all $m=1, \dots, k-1$ we have $\h \neq a^m \h$ and $\h \cap a^m \h \neq \emptyset$. Then the collection $H_k:= \{\h, a\h, \ldots, a^{k-1}\h\}$ is invariant under $a$. Also each pair of hyperplanes in $H_k$ intersects, therefore by the Helly property for hyperplanes they all have a common point of intersection which is invariant under $a$. This is again not possible since $\h$ separates $A$ and $B$. Therefore, there exists $r,s \in \{1,2, \ldots, d\}$ such that $a^r\h$ and $b^s\h$ are disjoint from $\h$ and contained in different halfspaces determined by $\h$. Let $\hs$ be the halfspace containing $A$. Then we have $b^s \hs \subset a^r \hs$. See Figure~\ref{fig:generating hyperbolic isometry}. Thus $a^{-r}b^s$ is a hyperbolic isometry of length at most $2\dimX$ in $a$ and $b$.

\begin{figure}[h]
\includegraphics[scale=0.25]{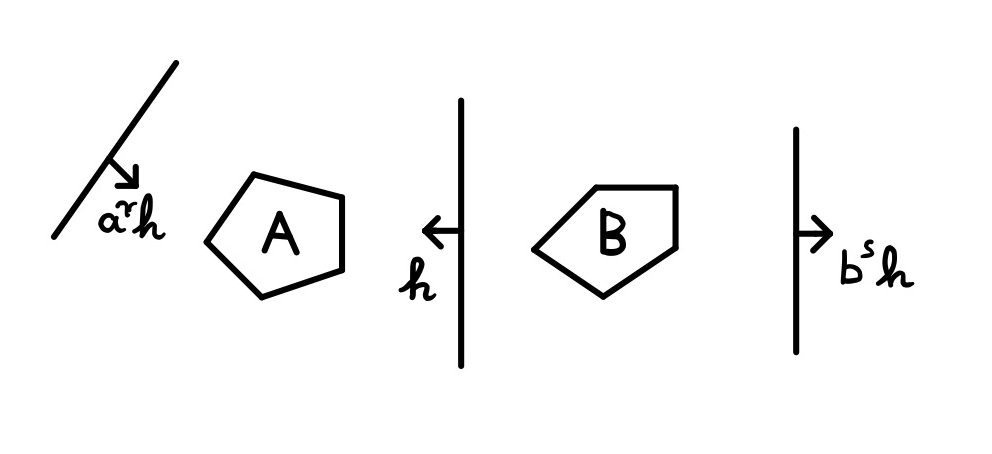}\caption{The hyperplanes $a^r\h$ and $b^s\h$ are disjoint from $\h$.}\label{fig:generating hyperbolic isometry}
\end{figure}

If there is no hyperplane separating $A$ and $B$, then $\Hull(A)$ and $\Hull(B)$ intersect in a non-empty set. Therefore, by \Cref{lem:cHull fixed}, $a^k$ and $b^k$ fix $\Hull(A) \cap \Hull(B)$, where $k = \text{lcm}\{1, 2, \ldots, \dimX \}$. 
\end{proof}

We will now prove \Cref{prop:elliptic hyperbolic}. Recall that the \emph{free Burnside group} $B(m,n)$ is the quotient of the free group on $m$ generators by the normal subgroup generated by the $n^{th}$ powers of all the elements. If $m\geq2$ and $n$ is sufficiently large, then $B(m,n)$ is infinite \cite{NovikovAdjan68}. However, for small values of $n$, some are known to be finite. 
For instance, $B(2,2) \cong \mathbb{Z}/2\mathbb{Z} \oplus \mathbb{Z}/2\mathbb{Z}$ and $B(2,6)$ are known to be finite \cite{Hall}. In order to extend our proof of Proposition~\ref{prop:elliptic hyperbolic} for cube complexes of dimension $\dimX \geq 4$, we would need $B(2,k)$, where $k = \text{lcm}\{1,2, \ldots, d\}$, to be finite. However, already for $\dimX = 4$, it is not known whether $B(2, 12)$ is finite or not.  

\ellipticHyperbolic

\begin{proof}
Let $F(\alpha,\beta)$ be the free group on two generators, generated by $\alpha$ and $\beta$. First assume $X$ is 2-dimensional. Let $K$ be the kernel of the map from $F(\alpha,\beta)$ to $B(2,2)$. Then $K$ is finitely generated by $\alpha^2, \beta^2, \alpha \beta^2\alpha^{-1}, \beta\alpha^2\beta^{-1}$. 
Let $w_1 = \alpha, w_2 = \beta, w_3 = \alpha \beta \alpha^{-1}, w_4 = \beta\alpha \beta^{-1}$. Then $K$ is generated by $w_1^2, w_2^2, w_3^2, w_4^2$ and each $w_i$ has length at most 3 in $F(\alpha,\beta)$. Let $\phi \colon F(\alpha,\beta) \to \la a, b \ra$ be the map that sends $\alpha$ to $a$ and $\beta$ to $b$. Then $\overline{w_i}:=\phi(w_i)$ for $ 1 \leq i \leq 4$ also has length at most 3 in $\la a, b \ra$. 
The following is a schematic: 
\[ \xymatrix{ 1 \ar[r]	&K	\ar@{^{(}->}[r] \ar[d]	&F(\alpha,\beta) \ar@{->>}[r] \ar[d]^{\phi}	&B(2,2) \ar[r] \ar[d]	&1\\
			1 \ar[r]	&\phi(K)	\ar@{^{(}->}[r] 	&\la a, b \ra \ar@{->>}[r] &(\text{finite}) \ar[r] &1} \]
 If $\overline{w_i}$ is a hyperbolic isometry of $X$ for some $ 1 \leq i \leq 4$, then we are done. So suppose each $\overline{w_i}$ is elliptic in $X$. For $i \neq j$, by Lemma~\ref{lem:separatedFixSetHyperbolic}, either  $\la \overline{w_i}^2, \overline{w_j}^2 \ra$ fixes $\Hull(\Fix(\overline{w_i})) \cap \Hull(\Fix(\overline{w_j}))$, or there exists a hyperbolic isometry in $\la \overline{w_i}, \overline{w_j} \ra$ of length at most 4 in $\overline{w_i}, \overline{w_j} $, so of length at most $4\cdot3 = 12 =: L$ in $\la a, b \ra$. Suppose the latter happens for each pair $\overline{w_i}, \overline{w_j}$. By Helly's property for cubically convex sets, there exists a point $x \in \bigcap \Hull(\Fix(\overline{w_i}))$, which is fixed by each $\overline{w_i}^2$ and hence by $\phi(K)$. Since $K$ is a finite index subgroup of $F(\alpha,\beta)$, $\phi(K)$ is a finite index subgroup of $\la a, b \ra$. Thus $\la a, b \ra$ has a global fixed point in $X$.

If $X$ is 3-dimensional, then we consider the group $B(2,6)$ instead of $B(2,2)$ which is also finite. Let $M >0$ be the maximum length in $\alpha,\beta$ of the elements in the smallest finite generating set of $K$. Then as above, we can either find a hyperbolic isometry of length at most $6M = L$ in $\la a, b \ra$ or $\la a, b \ra$ fixes a point in $X$.  
\end{proof}

    \section{Cubical groups generated with torsion} \label{sec:torsion2D}
In this section, $X$ will be a CAT(0) \emph{square} complex. We consider a finitely generated group $G$ acting \emph{without global fixed point} on $X$. 
In \cite{KarSageev}, the authors restrict their attention to free actions.  Freeness implies that every element of $G$ acts as a hyperbolic isometry of $X$. 
If the action is not free, then a given generating set may consist partially or entirely of elements acting elliptically on $X$. We will use Proposition~\ref{prop:elliptic hyperbolic} to construct a hyperbolic isometry from elliptic isometries and then build on the proof of Main Theorem in \cite{KarSageev}.

Before proving \Cref{thm:2Dtorsion}, we recall the basic case of groups acting on trees.

\begin{lem}
	\label{lem:treeFreeSemigroup}
	Let $S$ be a finite collection of isometries of a simplicial tree $T$.  
	Either $\abrackets{S}$ stabilizes a point or a line, or $\abrackets{S}$ contains a 4-short free semigroup.
\end{lem}
\begin{proof}
	Suppose $S$ contains a hyperbolic isometry $g$ of $T$ with axis $\ell$.  If $S \subset \stab(\ell)$ then we are done.  
	Otherwise, there exists $s \in S$ such that the axis $s\ell \neq \ell$.  By \cite[Proposition~10]{KarSageev}, $g^{\pm 1}$ and $sg^{\pm 1}s^{-1}$ generate a free semigroup.
	If $S$ consists only of elliptic isometries of $T$, then we may assume there exists elements $a, b \in S$ that do not fix a common point or else we are done.  By \cite[I~Proposition~26]{Serre}, the element $ab$ is a hyperbolic isometry of $T$.  We may repeat the argument above with $g = ab$.
\end{proof}

We use \Cref{prop:elliptic hyperbolic} to generalize \cite[Proposition~15]{KarSageev} to actions that are not free.

\begin{prop}\label{prop:collection of isometries}
Let $S$ be a finite collection of isometries of $X$, and suppose $S$ contains at least one hyperbolic isometry, and at least one elliptic isometry. Then either
\begin{enumerate}
	\item $\abrackets{S}$ contains a 50-short free semigroup, or 
	\item $\abrackets{S}$ stabilizes a flat or line in $X$. 
\end{enumerate}
\end{prop}

\begin{proof}
Let $S = S_0 \sqcup S_{elliptic}$ where $S_0$ is the set of hyperbolic isometries, 
and $S_{elliptic}$ is the set of elliptic isometries.
Let $S_1 = S_0 \cup \{ese^{-1}\mid s\in S_0, e\in S_{elliptic}\}$ and $S_2 = S_1 \cup \{ese^{-1} \mid s\in S_1, e\in S_{elliptic}\}$. 
Each of $S_0, S_1,$ and $S_2$ is a finite collection of hyperbolic isometries of $X$ with $S$-length at most 5.
Assume no pair of word of length at most $50$ in $S$ generates a free semigroup. 
In this case, we will find a Euclidean subcomplex that is stabilized by $S$.

Words in $S_i$ for $i = 0,1,2$ have $S$-length at most 5.  By assumption there does not exist a 50-short free semigroup with respect to $S$, so there is no 10-short free semigroup with respect to any $S_i$. By \cite[Theorem 1]{KarSageev}, there exists a minimal $S_2$-invariant subspace $F_2$, which is an isometrically embedded Euclidean plane or line (with respect to the CAT(0) metric). Since $S_1\subset S_2$, $F_2$ is also stabilized by $S_1$. Let $F_1$ be a minimal subspace of $F_2$ stabilized by $S_1$. Similarly, let $F_0$ be a minimal subspace of $F_1$ stabilized by $S_0$. We have $F_0\subseteq F_1 \subseteq F_2$. Since the three Euclidean spaces have dimensions $1$ or $2$, either $F_0 = F_1$, or $F_0\subsetneq F_1 = F_2$. For $i=1,2,3$, we denote the subcomplex $\Hull (F_i)$ by $E_i$.

First suppose that $F_0=F_1$. Then we also have that their cubical convex hulls $E_0=E_1$ are equal.
If $S_{elliptic}\subset \stab(E_1)$, we are done. 
Suppose there exists $e\in S_{elliptic}$ such that $eE_1\neq E_1$.  
If the subcomplexes $eE_1$ and $E_1$ are not parallel, 
then there exists a hyperplane $\h \in \Hyp (E_1)$ such that $e\h \notin \Hyp(E_1)$. 
Since $\h$ intersects $E_1$, which is a minimal Euclidean complex stabilized by $S_1$, 
there exists $g\in \langle S_1\rangle$ such that $\h \in \sk(g)$. 
It follows that $e\h\in e\sk(g) = \sk(ege^{-1})$. Since $ege^{-1}\in \langle S_1\rangle$
and $S_1$ stabilizes $E_1$, $e\h$ must intersect $E_1$. 
Thus $eE_1$ and $E_1$ must be parallel. 
Since $X$ is $2$-dimensional we get $\dim E_1 = \dim eE_1 = 1$.
Since $E_1$ is a subcomplex it must be a combinatorial line. 
The parallel subcomplex of $g$ is isometric to $E_1\times T$ and is stabilized by $S$ by repeating the above argument for every element $e \in S_{elliptic}$. 
Moreover, $\abrackets{S}$ preserves the product structure of $E_1 \times T$.
By \Cref{lem:treeFreeSemigroup}, $\abrackets{S}$ stabilizes a point or line $\ell$ in its action on $T$.  
Thus, $\abrackets{S}$ stabilizes the line $E_1$ or flat $E_1 \times \ell$.

Now consider the case where $F_0\subsetneq F_1 = F_2$. We must have $\dim F_0 = 1$ and $\dim F_1 =2$. 
Since $X$ is $2$-dimensional, any Euclidean 2-plane is equal to its cubical convex hull. Hence, we have $F_1 = E_1 = F_2 = E_2$.
We repeat the argument from the previous paragraph for $E_1, E_2$ in the place of $E_0, E_1$. 
We conclude that $S_{elliptic} \subset \stab(E_2)$, as otherwise we get a contradiction with the fact that $\dim E_1 = 2$.
\end{proof}

We are now ready to prove \Cref{thm:2Dtorsion} and \Cref{cor:Proper}.

\ThmA*

\begin{proof}
Let $S$ be a finite generating set for $G$. 
If $S$ contains no hyperbolic isometries then by Proposition~\ref{prop:elliptic hyperbolic} we may replace $S$
with a new generating set containing a hyperbolic isometry whose $S$-length is $\leq 12$. 
Thus without loss of generality we may assume that $S$ contains at least one hyperbolic isometry.
If $S$ contains no elliptic isometries, then we reduce to the Main Theorem of \cite{KarSageev}.
Suppose that $S$ contains at least one elliptic isometry. 
By Proposition~\ref{prop:collection of isometries}, either there exist a pair of words of length at most 50 in $S$ 
that freely generate a free semigroup or $S$ stabilizes a flat or line in $X$.
When $G$ has uniform exponential growth we get $w(G) \geq \sqrt[600]{2}$. 
\end{proof}

For proper actions on CAT(0) square complexes we can upgrade stabilizing a flat or line to being virtually abelian.

\CorProper
\begin{proof}
Suppose $G$ does not have uniform exponential growth with $w(G) \geq \sqrt[600]{2}$.
By \Cref{thm:2Dtorsion}, $G$ stabilizes a flat or line $F$. 
Consider the image $\bar G$ of $G$ in $\Isom(F)$. 
Since the action of $G$ on $X$ is proper, the action on $F$ is also proper, and so $\bar G$ is a discrete subgroup of $\Isom(F)$. 
Also, the properness of the action implies that $K =\ker (G\to \bar G)$ is finite. 
By Bierberbach's Theorem (\Cref{thm:Bieberbach}), the quotient $\bar G$ is virtually abelian. Thus, $G$ is finite-by-(virtually abelian) and hence virtually abelian.   For completeness we include a proof of this in \Cref{lem:finite-by-abelian} below.
\end{proof}

\begin{lem}
	\label{lem:finite-by-abelian} 
	Let $G$ be a finitely generated finite-by-(virtually abelian) group, i.e.\ there exists a finite normal subgroup $K<G$ such that $G/K$ is virtually abelian. Then $G$ is virtually abelian.
\end{lem}
\begin{proof} 
	Let $A$ be a maximal rank free abelian subgroup of $G/K$ and let $H$ be the preimage of $A$ in $G$.  Then $H$ is a finite index subgroup of $G$ and fits into the following short exact sequence:
	\[
		1\to K\to H\to A\to 1.
	\] 
	Let $\pi:H\to A$ denote the projection.
The group $H$ acts on $K$ by conjugation, i.e. there is a homomorphism $H\to \text{Aut} (K)$, so we can pass to a further finite index subgroup $H'$ whose action on $K$ is trivial, i.e. we have the following central extension
\[
1\to K'\to H'\to A'\to 1
\] 
where $K' = K\cap H'$ and $A' = \pi(H')<A$. Let $\{a_1,\dots, a_n\}$ be a minimal set of generators of $A'$, and let $\{h_1,\dots, h_n\}$ be a set of elements of $H$ such that $\pi(h_i) = a_i$ for all $1\leq i\leq n$. We will show that $\langle h_1^{m_1}, \dots, h_n^{m_n}\rangle = \Z^n$ for some $m_1,\dots, m_n\in \N$. This follows from the following claim: for any $g,h\in H'$ with $[\pi(g), \pi(h)]=1$, there exists $m\in \N$ such that $[g, h^m] = 1$. To see that the claim is true, note that for every $m\in\N$ we have $[g,h^m]\in K'$. Since $K'$ has only finitely many elements there exists distinct $m_1,m_2$ such that $gh^{m_1}g^{-1}h^{-m_1} = gh^{m_2}g^{-1}h^{-m_2}$, which implies that $h^{m_1-m_2}g^{-1} = g^{-1}h^{m_1-m_2}$. 

It remains to show that $H'':=\langle h_1^{m_1}, \dots, h_n^{m_n}\rangle$ has finite index in $H'$. Note that $A'':=\langle a_1^{m_1}, \dots, a_n^{m_n}\rangle$  has finite index in $A'$. The finite index subgroup $\phi^{-1}(A'')$ of $H'$ is generated by $\{ h_1^{m_1}, \dots, h_n^{m_n}\}\cup K$, and so is isomorphic to $K\times H''$.
\end{proof}

In general, stabilizing a flat is far from sufficient to show that a group is virtually abelian.  
The following example was brought to our attention by Talia Fern\'{o}s. 

\begin{ex}
	Let $G = \Z^2 \oplus R$ where $R$ is the Grigorchuk group, which has intermediate growth.  $R$ acts faithfully on a tree $T$, with a single global fixed point $v$, so $G$ acts faithfully on the universal cover of the wedge sum a torus and $T$ along the vertex $v$.  The torus lifts to a 2-flat stabilized by $G$, but $G$ does not act faithfully on this flat.  Moreover, $G$ is neither virtually abelian nor contains a free semigroup because it has intermediate growth.
\end{ex}

Nevertheless, by studying the interactions between vertex stabilizers it is sometimes possible to show that certain groups acting improperly on a CAT(0) square complex may still satisfy the conclusion of \Cref{cor:Proper}.

    \section{Improper actions and locally uniform exponential growth}\label{sec:Applications}

It is not known whether all acylindrically hyperbolic groups have uniform exponential growth. 
However, they may contain finitely generated exponentially growing subgroups without uniform exponential growth.  For example, this is the case for the free product of Wilson's group with itself.  
In this section, we will show that by understanding vertex stabilizers it is possible to use Theorem~\ref{thm:2Dtorsion} to prove locally uniform exponential growth results for certain acylindrically hyperbolic groups that also act on cube complexes.  In each case, we make use of the following.

\improperSquareCplx

\begin{proof}
	If $H$ acts without global fixed point on $X$, then by Theorem~\ref{thm:2Dtorsion}, either it contains a uniformly short free semigroup, 
	or it stabilizes a flat or line.
	If $H$ has a fixed point in $X$, then it is a finitely generated subgroup of one of the vertex groups, so either is virtually abelian or has uniform exponential growth bounded by $w_0$.  
\end{proof}

\subsection{Higman group} 
The \emph{Higman group} \cite{Higman}, $H$, is given by the following presentation.  
$$ H := \abrackets{a_i \mid a_i(a_{i + 1})a_i^{-1} = a_{i +1}^2  }_{i \in \Z/4\Z}  $$
This presentation gives a decomposition of $H$ as a square of groups with the following local groups. 
Each vertex group is a copy of $BS(1,2)$.
Each edge group is a copy of $\Z$.
Each 2-cell group is trivial.  
This decomposition gives a cocompact action of $H$ on a CAT(0) square complex $X$, whose vertex stabilizers are the groups mentioned above.  
Martin used this structure to show that certain generalizations of the Higman group act acylindrically hyperbolically \cite[Theorem~B]{Martin} on CAT(0) square complexes. The Higman group itself is acylindrically hyperbolic coming from its structure as a free product with amalgamation and \cite{MinasyanOsin}.

\begin{thm}
	\label{cor:Higman}
	Let $G$ be any finitely generated subgroup of the Higman group $H$. Then either $G$ is cyclic or $G$ has uniform exponential growth with $w(G) \geq \sqrt[600]{2}$ . 
\end{thm}

To understand exponential growth in the Higman group, we first show uniform exponential growth of finitely generated subgroups of Baumslag-Solitar groups.  Uniform exponential growth of solvable Baumslag-Solitar groups follows from work of Bucher and de la Harpe \cite{BucherHarpe}, however they do not address subgroups.
	
\begin{lem}[Baumslag-Solitar groups]
	\label{lem:BaumslagSolitar}
	Any finitely generated subgroup of a Baumslag-Solitar group $BS(1,m)$ where $m \neq \pm 1$, is either cyclic or has uniform exponential growth bounded by $\sqrt[4]{2}$.
\end{lem}
	
\begin{proof}
	Let $S$ be any finite collection of elements of
	the Baumslag-Solitar group with presentation $\abrackets{a,t \mid tat^{-1} = a^m}$. Let $T$ be the Bass-Serre tree for $BS(1,m)$ with $\Z$ vertex and edge groups.  This tree can be obtained from the Cayley complex by collapsing in the $a$-direction.
	
	By \Cref{lem:treeFreeSemigroup}, either $G = \abrackets{S}$ contains a 4-short free semigroup or stabilizes a point or line in $T$.  
	If $G$ stabilizes a point then $G$ is cyclic, so assume that $G$ stabilizes a line $\ell$.  It suffices to show that $\stab(\ell)$ is cyclic.
	
	Every element of $BS(1,m)$ can be written in the form $h = a^pt^q$ because $ta = a^mt$.  We claim elements of the form $a^p$ cannot stabilize $\ell$.  Indeed, if $a^p$ stabilizes $\ell$ then it would fix the line pointwise because such elements fix a vertex in $T$.  Vertex stabilizers are conjugates of $a$, so segments of length $n$ can only be fixed pointwise by elements that are powers of $a^{mn}$.  Taking $n$ larger than $p$ gives a contradiction.
	If $q \neq 0$ then $h$ is also a hyperbolic isometry of $T$.  The only hyperbolic isometries that will stabilize the axis of $g$ are 
	roots and powers of $g$.  The axes of any other hyperbolic isometry $h$ will diverge from $\ell$ in the tree $T$, so $h$ cannot stabilize $\ell$.  It follows that $\stab(\ell)$ is cyclic.
\end{proof}

With this, we are ready to address the Higman group.  

\begin{proof}[Proof of \Cref{cor:Higman}]
	By \Cref{cor:improper2D} and \Cref{lem:BaumslagSolitar}, any finitely generated subgroup $G \leq H$ of the Higman group either has $w(G) \geq \sqrt[600]{2}$, or is virtually abelian,
	or stabilizes a flat or line in $X$.
	Let $E$ be the line or flat in $X$ stabilized by $G$.  We have homomorphism $\pi:G \to \Isom(E)$ where the image $\bar{G}$ is virtually abelian.  
	The kernel is contained in $\bigcap_{p \in E} \stab(p)$, so it is cyclic. 
	If $\dim(E) = 2$ then $\ker(\pi)$ is trivial because the 2-cell groups are trivial.  If $\dim(E) = 1$ then $E$ is contained in the 1-skeleton of $X$.  Stabilizers of adjacent edges have trivial intersection by \cite[Lemma~2.1]{MartinHigman}, so we again have $\ker(\pi)$ is trivial. 
	Therefore, $G \cong \bar{G}$, and moreover $G$ is cyclic, since Higman group is torsion-free and does not contain $\Z^2$ as a subgroup \cite[Proposition~4.15]{MartinHigman}.
\end{proof}
	
\subsection{Triangle-free Artin groups}
Artin groups generalize the braid group. They admit presentations corresponding to finite labeled graphs where each label $m$ is at least two.   
Vertices correspond to generators and an edge labeled by $m$ joining vertices $a$ and $b$ corresponds to the relation
$$ \underbrace{aba \cdots}_{m} = \underbrace{bab \cdots}_{m} .$$
\emph{Triangle-free} Artin groups are those whose defining graphs have girth $\geq 4$.
\emph{Spherical} Artin groups are those whose quotient Coxeter group is finite.  This quotient is obtained by imposing that all generators have order 2. An Artin group is \emph{FC-type} when every clique in the defining graph is associated to a spherical Artin subgroup. An Artin group is \emph{$2$-dimensional} if every spherical Artin subgroup has rank at most $2$. It is easy to see that an Artin group is triangle-free if and only if it is $2$-dimensional FC-type.

Many subclasses of Artin groups exhibit properties of nonpositive curvature and have attracted much attention in recent years.
Indeed, FC-type Artin groups are acylindrically hyperbolic by Chatterji and Martin \cite[Theorem~1.2]{ChatterjiMartin}. In recent work, Martin and Przytycki exploit an improper action of FC-type Artin groups on CAT(0) cube complexes in order to prove the strong Tits alternative \cite{MartinPrzytycki}. 

Charney and Davis showed an Artin group $A$ is FC-type if and only if its Deligne complex $\Deligne$ is a CAT(0) cube complex.  Moreover, they showed that $\Deligne$ is a $K(\pi,1)$ space for $A$ so it has the same cohomological dimension as the Artin group \cite[Theorem~4.3.5]{CharneyDavisCAT0} (see also \cite{MartinPrzytycki}).
In particular, the Deligne complex of triangle-free Artin groups is a CAT(0) square complex. 
The action of any FC-type Artin group $A$ on its Deligne complex has vertex stabilizers that are conjugates of standard parabolic subgroups, which correspond to subgraphs of the defining graph of $A$.

\begin{thm}\label{thm:FCtypeArtin}
	Let $G$ be any finitely generated subgroup of a triangle-free Artin group $A$.  Either $G$ is virtually abelian or it has uniform exponential growth with $w(G) \geq \sqrt[600]{2}$.  
\end{thm}

The base case of the theorem is the following.
\begin{lem}[Rank $2$ Artin groups]
	\label{lem:Rank2Artin}
	Any finitely generated subgroup of a rank $2$ Artin group is either virtually abelian or has uniform exponential growth bounded by $\sqrt[4]{2}$.
\end{lem}
\begin{proof}
	Any rank $2$ Artin group $A$ acts geometrically on a CAT(0) square complex isometric to $\R \times T$ where $T$ is a simplicial tree \cite[Theorem~5.1]{HJP:cube2DArtin}  (a similar complex is also described in \cite{BradyMcCammond}). 
	Also, $A$ is torsion-free \cite{Deligne}. Therefore, any subgroup of $A$ acts freely on $\R \times T$.
	The bounds on growth follow from \cite[Theorem~1]{KarSageev} in the special case where $X$ is isometric to $\R \times T$.
\end{proof}

\begin{proof}[Proof of Theorem~\ref{thm:FCtypeArtin}] 
The Deligne complex $\Deligne$ of a triangle-free Artin group is a CAT(0) square complex. Every edge in $\Deligne$ lies in a square and every square in $\Deligne$ has one vertex with trivial stabilizer joined to two vertices with cyclic stabilizers that are conjugates of subgroups generated by two distinct standard generators, and one vertex that is a conjugate of a rank $2$ standard parabolic subgroup (see \cite{CharneyDavisCAT0,MartinPrzytycki}).
 
Suppose that the exponential growth rate $w(G) < \sqrt[600]{2}$.  
If $G$ has a global fixed point in $\Deligne$ then $G$ is virtually abelian by \Cref{lem:Rank2Artin}. 
Hence, by \Cref{cor:improper2D}, $G$ stabilizes a flat or line $E \subset \Deligne$.  
Therefore, the image, $\bar G$, of $G$ in $\Isom(E)$ is virtually abelian. The kernel $\ker(G\to \bar G)$ has a subgroup $K$ of index at most $2$ that pointwise stabilizes the cubical convex hull of $E$. The group $K$ is contained in the pointwise stabilizer $\bigcap_{x\in \Hull(E)^{(0)}} \stab(x)$ of $\Hull (E)$ in $G$, which we will show is trivial. If $\Hull(E)$ contains any vertex with trivial stabilizer, then $\bigcap_{x\in \Hull(E)^{(0)}} \stab(x)$ is clearly also trivial. Suppose that $E$ is a combinatorial line that does not contain any vertices with trivial stabilizers. 
Such a line must contain a vertex with cyclic stabilizer, so $\ker(G \to \bar{G})$ is (possibly trivial) cyclic group.  Hence, $G$ is a cyclic-by-(virtually $\Z$) group, which is virtually abelian. 
\end{proof}

\bibliographystyle{alpha}
\bibliography{UEG_Bib}

\begin{thebibliography}{{Mar}15}

\bibitem[Alp02]{Alperin}
Roger~C. Alperin.
\newblock Uniform growth of polycyclic groups.
\newblock volume~92, pages 105--113. 2002.
\newblock Dedicated to John Stallings on the occasion of his 65th birthday.

\bibitem[AN02]{AlperinNoskov}
Roger~C. Alperin and Guennadi~A. Noskov.
\newblock Uniform growth, actions on trees and {${\rm GL}_2$}.
\newblock In {\em Computational and statistical group theory ({L}as {V}egas,
  {NV}/{H}oboken, {NJ}, 2001)}, volume 298 of {\em Contemp. Math.}, pages 1--5.
  Amer. Math. Soc., Providence, RI, 2002.

\bibitem[ANS19]{ANS}
Carolyn~R. Abbott, Thomas Ng, and Davide Spriano.
\newblock Hierarchically hyperbolic groups and uniform exponential growth.
\newblock {\em arXiv:1909.00439}, 2019.

\bibitem[Bal95]{BallmannLectures}
Werner Ballmann.
\newblock {\em Lectures on spaces of nonpositive curvature}, volume~25 of {\em
  DMV Seminar}.
\newblock Birkh\"{a}user Verlag, Basel, 1995.
\newblock With an appendix by Misha Brin.

\bibitem[BCG11]{BCG}
G.~Besson, G.~Courtois, and S.~Gallot.
\newblock Uniform growth of groups acting on {C}artan-{H}adamard spaces.
\newblock {\em J. Eur. Math. Soc. (JEMS)}, 13(5):1343--1371, 2011.

\bibitem[BdlH00]{BucherHarpe}
Michelle Bucher and Pierre de~la Harpe.
\newblock Free products with amalgamation, and {HNN}-extensions of uniformly
  exponential growth.
\newblock {\em Mat. Zametki}, 67(6):811--815, 2000.

\bibitem[BF02]{BF}
Mladen Bestvina and Koji Fujiwara.
\newblock Bounded cohomology of subgroups of mapping class groups.
\newblock {\em Geom. Topol.}, 6:69--89 (electronic), 2002.

\bibitem[BF18]{BreuillardFujiwara}
Emmanuel Breuillard and Koji Fujiwara.
\newblock On the joint spectral radius for isometries of non-positively curved
  spaces and uniform growth.
\newblock {\em arXiv:1804.00748}, 2018.

\bibitem[BH99]{BridsonHaefliger}
Martin~R. Bridson and Andr\'{e} Haefliger.
\newblock {\em Metric spaces of non-positive curvature}, volume 319 of {\em
  Grundlehren der Mathematischen Wissenschaften [Fundamental Principles of
  Mathematical Sciences]}.
\newblock Springer-Verlag, Berlin, 1999.

\bibitem[BM00]{BradyMcCammond}
Thomas Brady and Jonathan~P. McCammond.
\newblock Three-generator {A}rtin groups of large type are biautomatic.
\newblock {\em J. Pure Appl. Algebra}, 151(1):1--9, 2000.

\bibitem[Bro87]{Brown}
Kenneth~S. Brown.
\newblock Trees, valuations, and the {B}ieri-{N}eumann-{S}trebel invariant.
\newblock {\em Invent. Math.}, 90(3):479--504, 1987.

\bibitem[CD95]{CharneyDavisCAT0}
Ruth Charney and Michael~W. Davis.
\newblock The {$K(\pi,1)$}-problem for hyperplane complements associated to
  infinite reflection groups.
\newblock {\em J. Amer. Math. Soc.}, 8(3):597--627, 1995.

\bibitem[CM19]{ChatterjiMartin}
Indira Chatterji and Alexandre Martin.
\newblock A note on the acylindrical hyperbolicity of groups acting on {${\rm
  CAT}(0)$} cube complexes.
\newblock In {\em Beyond hyperbolicity}, volume 454 of {\em London Math. Soc.
  Lecture Note Ser.}, pages 160--178. Cambridge Univ. Press, Cambridge, 2019.

\bibitem[Del72]{Deligne}
Pierre Deligne.
\newblock Les immeubles des groupes de tresses g\'en\'eralis\'es.
\newblock {\em Invent. Math.}, 17:273--302, 1972.

\bibitem[Del91]{DelzantTwoGen}
Thomas Delzant.
\newblock Sous-groupes \`a deux g\'{e}n\'{e}rateurs des groupes hyperboliques.
\newblock In {\em Group theory from a geometrical viewpoint ({T}rieste, 1990)},
  pages 177--189. World Sci. Publ., River Edge, NJ, 1991.

\bibitem[DGO17]{DahmaniGuirardelOsin}
F.~Dahmani, V.~Guirardel, and D.~Osin.
\newblock Hyperbolically embedded subgroups and rotating families in groups
  acting on hyperbolic spaces.
\newblock {\em Mem. Amer. Math. Soc.}, 245(1156):v+152, 2017.

\bibitem[DKL19]{DeyKapovichLiu}
Subhadip Dey, Michael Kapovich, and Beibei Liu.
\newblock Ping-pong in {H}adamard manifolds.
\newblock {\em M\"{u}nster J. Math.}, 12(2):453--471, 2019.

\bibitem[dlH02]{delaharpe}
Pierre de~la Harpe.
\newblock Uniform growth in groups of exponential growth.
\newblock In {\em Proceedings of the {C}onference on {G}eometric and
  {C}ombinatorial {G}roup {T}heory, {P}art {II} ({H}aifa, 2000)}, volume~95,
  pages 1--17, 2002.

\bibitem[DT06]{DunfieldThurston}
Nathan~M Dunfield and Dylan~P Thurston.
\newblock A random tunnel number one 3–manifold does not fiber over the
  circle.
\newblock {\em Geom. Topol.}, 10(4):2431--2499, 2006.

\bibitem[EMO05]{EMO}
Alex Eskin, Shahar Mozes, and Hee Oh.
\newblock On uniform exponential growth for linear groups.
\newblock {\em Invent. Math.}, 160(1):1--30, 2005.

\bibitem[GdlH97]{GrigorchukHarpeGrowth}
R.~Grigorchuk and P.~de~la Harpe.
\newblock On problems related to growth, entropy, and spectrum in group theory.
\newblock {\em J. Dynam. Control Systems}, 3(1):51--89, 1997.

\bibitem[GdlH01]{GrigorchukHarpe}
Rostislav~I. Grigorchuk and Pierre de~la Harpe.
\newblock One-relator groups of exponential growth have uniformly exponential
  growth.
\newblock {\em Mat. Zametki}, 69(4):628--630, 2001.

\bibitem[Gro87]{Gromov}
M.~Gromov.
\newblock {\em Hyperbolic Groups}, pages 75--263.
\newblock Springer New York, New York, NY, 1987.

\bibitem[{Hag}07]{Haglund:semisimplicity}
Fr{\'e}d{\'e}ric {Haglund}.
\newblock {Isometries of CAT(0) cube complexes are semi-simple}.
\newblock {\em arXiv:0705.3386}, 2007.

\bibitem[Hag20]{Hagen2020}
Mark Hagen.
\newblock Large facing tuples and a strengthened sector lemma.
\newblock {\em arXiv:2005.09536}, 2020.

\bibitem[Hal58]{Hall}
Marshall Hall, Jr.
\newblock Solution of the {B}urnside problem for exponent six.
\newblock {\em Illinois J. Math.}, 2:764--786, 1958.

\bibitem[Hig51]{Higman}
Graham Higman.
\newblock A finitely generated infinite simple group.
\newblock {\em J. London Math. Soc.}, 26:61--64, 1951.

\bibitem[HJP16]{HJP:cube2DArtin}
Jingyin Huang, Kasia Jankiewicz, and Piotr Przytycki.
\newblock Cocompactly cubulated 2-dimensional {A}rtin groups.
\newblock {\em Comment. Math. Helv.}, 91(3):519--542, 2016.

\bibitem[HK05]{HruskaKleiner}
G.~Christopher Hruska and Bruce Kleiner.
\newblock Hadamard spaces with isolated flats.
\newblock {\em Geom. Topol.}, 9:1501--1538, 2005.
\newblock With an appendix by the authors and Mohamad Hindawi.

\bibitem[HP98]{HaglundPaulin:Wallspace}
Fr\'{e}d\'{e}ric Haglund and Fr\'{e}d\'{e}ric Paulin.
\newblock Simplicit\'{e} de groupes d'automorphismes d'espaces \`a courbure
  n\'{e}gative.
\newblock In {\em The {E}pstein birthday schrift}, volume~1 of {\em Geom.
  Topol. Monogr.}, pages 181--248. Geom. Topol. Publ., Coventry, 1998.

\bibitem[Hru05]{Hruska:GeometricInvariants}
G.~Christopher Hruska.
\newblock Geometric invariants of spaces with isolated flats.
\newblock {\em Topology}, 44(2):441--458, 2005.

\bibitem[Hua17]{Huang:QIraags}
Jingyin Huang.
\newblock Quasi-isometric classification of right-angled {A}rtin groups {I}:
  the finite out case.
\newblock {\em Geom. Topol.}, 21(6):3467--3537, 2017.

\bibitem[Jan20]{Jankiewicz}
Kasia Jankiewicz.
\newblock Lower bounds on cubical dimension of {C}$'$(1/6) groups.
\newblock {\em Proc. Amer. Math. Soc.}, 148(8):3293--3306, 2020.

\bibitem[Kap19]{KapovichMisha}
Michael Kapovich.
\newblock A note on properly discontinuous actions.
\newblock {\em https://www.math.ucdavis.edu/~kapovich/EPR/prop-disc.pdf}, 2019.

\bibitem[Kou98]{Koubi}
Malik Koubi.
\newblock Croissance uniforme dans les groupes hyperboliques.
\newblock {\em Ann. Inst. Fourier (Grenoble)}, 48(5):1441--1453, 1998.

\bibitem[KS19]{KarSageev}
Aditi Kar and Michah Sageev.
\newblock Uniform exponential growth for {CAT}(0) square complexes.
\newblock {\em Algebr. Geom. Topol.}, 19(3):1229--1245, 2019.

\bibitem[Mag30]{Magnus}
Wilhelm Magnus.
\newblock \"{U}ber diskontinuierliche {G}ruppen mit einer definierenden
  {R}elation. ({D}er {F}reiheitssatz).
\newblock {\em J. Reine Angew. Math.}, 163:141--165, 1930.

\bibitem[Man10]{Mangahas}
Johanna Mangahas.
\newblock Uniform uniform exponential growth of subgroups of the mapping class
  group.
\newblock {\em Geom. Funct. Anal.}, 19(5):1468--1480, 2010.

\bibitem[{Mar}15]{Martin}
Alexandre {Martin}.
\newblock {Acylindrical actions on CAT(0) square complexes}.
\newblock {\em arXiv:1509.03131}, 2015.

\bibitem[Mar17]{MartinHigman}
Alexandre Martin.
\newblock On the cubical geometry of {H}igman's group.
\newblock {\em Duke Math. J.}, 166(4):707--738, 2017.

\bibitem[MO15]{MinasyanOsin}
Ashot Minasyan and Denis Osin.
\newblock Acylindrical hyperbolicity of groups acting on trees.
\newblock {\em Math. Ann.}, 362(3-4):1055--1105, 2015.

\bibitem[MP19]{MartinPrzytycki}
Alexandre Martin and Piotr Przytycki.
\newblock Tits alternative for artin groups of type {FC}.
\newblock {\em arXiv:1906.07393}, 2019.

\bibitem[NA68]{NovikovAdjan68}
P.~S. Novikov and S.~I. Adjan.
\newblock Infinite periodic groups. {I, II, III}.
\newblock {\em Izv. Akad. Nauk SSSR Ser. Mat.}, 32:212--244, 251--524,
  709--731, 1968.

\bibitem[Osi03]{Osin:solvable}
D.~Osin.
\newblock The entropy of solvable groups.
\newblock {\em Ergodic Theory Dynam. Systems}, 23(3):907--918, 2003.

\bibitem[Osi16]{Osin}
D.~Osin.
\newblock Acylindrically hyperbolic groups.
\newblock {\em Trans. Amer. Math. Soc.}, 368(2):851--888, 2016.

\bibitem[Rua01]{Ruane:Dynamics}
Kim~E. Ruane.
\newblock Dynamics of the action of a cat(0) group on the boundary.
\newblock {\em Geometriae Dedicata}, 84(1):81--99, 2001.

\bibitem[Sag95]{Sageev:EndsNPC}
Michah Sageev.
\newblock Ends of group pairs and non-positively curved cube complexes.
\newblock {\em Proc. London Math. Soc. (3)}, 71(3):585--617, 1995.

\bibitem[Sag14]{SageevPCMI}
Michah Sageev.
\newblock {$\rm CAT(0)$} cube complexes and groups.
\newblock In {\em Geometric group theory}, volume~21 of {\em IAS/Park City
  Math. Ser.}, pages 7--54. Amer. Math. Soc., Providence, RI, 2014.

\bibitem[Ser03]{Serre}
Jean-Pierre Serre.
\newblock {\em Trees}.
\newblock Springer Monographs in Mathematics. Springer-Verlag, Berlin, 2003.
\newblock Translated from the French original by John Stillwell, Corrected 2nd
  printing of the 1980 English translation.

\bibitem[SW05]{SageevWise}
Michah Sageev and Daniel~T. Wise.
\newblock The {T}its alternative for {${\rm CAT}(0)$} cubical complexes.
\newblock {\em Bull. London Math. Soc.}, 37(5):706--710, 2005.

\bibitem[Thu97]{ThurstonBook}
William~P. Thurston.
\newblock {\em Three-dimensional geometry and topology. {V}ol. 1}, volume~35 of
  {\em Princeton Mathematical Series}.
\newblock Princeton University Press, Princeton, NJ, 1997.
\newblock Edited by Silvio Levy.

\bibitem[Wil04]{Wilson}
John~S. Wilson.
\newblock On exponential growth and uniformly exponential growth for groups.
\newblock {\em Invent. Math.}, 155(2):287--303, 2004.

\bibitem[Wis07]{Wise:antitorus}
Daniel~T. Wise.
\newblock Complete square complexes.
\newblock {\em Comment. Math. Helv.}, 82(4):683--724, 2007.

\bibitem[Xie07]{Xie}
Xiangdong Xie.
\newblock Growth of relatively hyperbolic groups.
\newblock {\em Proc. Amer. Math. Soc.}, 135(3):695--704, 2007.

\end{thebibliography}

\end{document}